\documentclass[preprint,12pt]{elsarticle}

\usepackage{amssymb}
\usepackage{graphicx}
\usepackage{mathptmx}      % use Times fonts if available on your TeX system

\usepackage{amssymb,amsfonts,amsmath,amscd,latexsym}
\usepackage[english]{babel}
\usepackage{geometry}
\usepackage{latexsym}
\usepackage{amssymb}
\usepackage[all]{xy}
\usepackage{amssymb}
\usepackage{latexsym}
\usepackage{graphicx}
\usepackage{appendix}
\usepackage[active]{srcltx}
\usepackage{srcltx}
\usepackage{xypic}
\usepackage{amsthm}

\usepackage[bbgreekl]{mathbbol}
\usepackage[active]{srcltx}
\newtheorem{theorem}{Theorem}[section]
\newtheorem{lemma}[theorem]{Lemma}
\newtheorem{proposition}[theorem]{Proposition}
\newtheorem{definition}[theorem]{Definition}

\newtheorem{remark}[theorem]{Remark}

\numberwithin{equation}{section}
\setcitestyle{square}
\newcommand{\HH}{\mathrm{HH}}
\newcommand{\Hc}{\mathrm{H}}
\newcommand{\res}{\mathrm{res}}

\newcommand{\Ima}{\mathrm{Im}}

\newcommand{\Ind}{\mathrm{Ind}}
\newcommand{\Res}{\mathrm{Res}}
\newcommand{\id}{\mathrm{id}}
\newcommand{\tr}{\mathrm{tr}}
\newcommand{\Ext}{\mathrm{Ext}}

\newcommand{\Fp}{\mathbb{F}_p}
\newcommand{\Z}{\mathbb{Z}}

\begin{document}

\begin{frontmatter}



\title{Bockstein homomorphisms for Hochschild cohomology of group algebras and of block algebras of finite groups}


\author{Constantin-Cosmin Todea}

%\address{Department of Mathematics, Technical University of Cluj Napoca, Str. G. Bari\c tiu, 25, 400027, Romania }
\address{Department of Mathematics, Technical University of Cluj-Napoca, Str. G. Baritiu 25,
 Cluj-Napoca 400027, Romania}

\ead{Constantin.Todea@math.utcluj.ro}
%\fntext[secondadr]{Department of Mathematics, Technical University of Cluj Napoca, Str. G. Bari\c tiu, 25, 400027, Romania}
\begin{abstract} We give an explicit approach for Bockstein homomorphisms of the Hochschild cohomology of a group algebra and of a block algebra of a finite group and we show some properties. To give  explicit definitions for these maps we use an additive decomposition and a Product Formula for the Hochschild cohomology of group algebras given by Siegel and Witherspoon in 1999. For $k$ an algebraically closed field of characteristic $p$ and $G$ a finite group we prove an additive decomposition and a Product Formula for the cohomology algebra of a defect group of a block ideal of $kG$ with coefficients in the source algebra of this block, and we define similar Bockstein homomorphisms.
\end{abstract}


\end{frontmatter}

% \linenumbers

\section{Introduction}
The Bockstein homomorphism in group cohomology is the connecting homomorphism in the long exact sequence associated  to some short exact sequence of coefficients. It appears in the Bockstein spectral sequence, which is a tool for comparing integral and mod $p$ cohomology ($p$ is a prime), and has applications for Steenrod operations. We use results from \cite{BenII}, regarding Bockstein maps in group cohomology and with some minor adjustments we adopt the same notations. For the rest of the paper, if otherwise not specified, let $G$ be a finite group, let $\Fp$ be the field with $p$ elements and $k$ be a field of characteristic $p$, with $p$ dividing the order of $G$. Let $r$ be a non-negative degree integer; since we work with more than one group  we denote by $\beta_G$ the Bockstein homomorphism
$$\beta_G:\Hc^r(G,\Fp)\rightarrow\Hc^{r+1}(G,\Fp)$$
and by $\widehat{\beta}_G$ the homomorphism
$$\widehat{\beta}_G:\Hc^r(G,\Fp)\rightarrow\Hc^{r+1}(G,\Z).$$

Another application of Bockstein homomorphism is that $\beta_G$ gives a structure of (cochain) DG-algebra to the cohomology algebra $\Hc^*(G,\Fp)$, that is $\beta_G$ is a differential (\cite[Lemma 4.3.2]{BenII}) and satisfies the graded Leibniz rule, (\cite[Lemma 4.3.3]{BenII}). In this paper we will call a map which satisfies this rule a graded derivation. Bockstein maps appear also in a celebrated theorem of Serre \cite[Theorem 4.7.3]{BenII} on zero product of Bocksteins, given originally in a different form. This theorem is an important ingredient  for proving a theorem of Carlson (or Quillen) which characterizes nilpotent elements of positive degree in Ext-algebras over finitely generated $kG$-modules by nilpotence of restrictions to elementary abelian subgroups. Further, this theorem can be used  for showing a remarkable theorem of Chouinard \cite{Cho} which states that a finitely generated $kG$-module $M$ is projective if and only if its restriction to every elementary abelian subgroup of $G$ is projective.

In the case of group algebras it is well known that Hochschild cohomology $\HH^*(kG)$ can be identified with the ordinary cohomology  $\Hc^*(G,kG)$, where $kG$ is viewed as left $kG$-module by conjugation. This observation goes back to Eilenberg and MacLane \cite{EilMac}. However the isomorphism as graded $k$-algebras was given more recently by Siegel and Witherspoon \cite[Proposition 3.1]{SieWith}. Here $\HH^*(kG)$ is an associative graded algebra with the usual Hochschild cup  product and $\Hc^*(G,kG)$ is an associative graded algebra with the product called \emph{cup product with respect to the pairing} $\mu$, where $\mu$ is the multiplication map in $kG$; see \cite[Sections 2, 3]{SieWith} for a general discussion regarding cup products in this context. For the rest of this paper, this identification allows us to say and have in mind Hochschild cohomology of group algebras $\HH^*(kG)$, but we will work with $\Hc^*(G,kG)$. In 
Section \ref{sec2} we will define  Bockstein homomorphisms for the Hochschild cohomology of group algebras $\HH^*(kG)$. In the first part of this section we define the Bockstein homomorphism for $\Hc^*(G,\Fp G)$. Using an additive decomposition and a Product Formula from \cite{SieWith}, in the main result of this section, Theorem \ref{thmBockderiv}, we extend this definition to $\Hc^*(G,kG)$ and we  prove that $\Hc^*(G,kG)$ is also a DG-algebra.

In Section \ref{sec3} we deal with block algebras of $kG$, where from now $k$ is an algebraically closed field of characteristic $p$. Cohomology of block algebras of finite groups was first introduced by Linckelmann in \cite{LiTr} and further studied in a series of papers \cite{LiVar}, \cite{LiQuill}. An important recent paper is \cite{LiVariso} where the author shows that the varieties of the Hochschild cohomology of a block algebra and its block cohomology are isomorphic, hence answering some questions raised by Pakianathan and Witherspoon \cite{PakWith}. As in the case of group algebras we identify the Hochschild cohomology of a block algebra $B$ of $kG$ with the cohomology of $G$ with coefficients in $B$, where $G$ acts by conjugation on $B$, that is $\HH^*(B)\cong \Hc^*(G,B)$. With this identification we define Bockstein homomorphisms $\bbbeta_{G,B}$ for $\Hc^*(G,B)$ in Definition \ref{defnBockHHbloc}.  Using a trick we obtain in Proposition \ref{prop31addithochbloc} an additive decomposition of $\Hc^*(G,B)$ which allows us to give an alternative, explicit  definition of the Bockstein homomorphisms $\bbbeta_{G,B}$ for $\Hc^*(G,B)$. After we recall the definition of the block cohomology which we denote by $\Hc^*(G,b,P_{\delta})$, where $b$ is the primitive central idempotent corresponding to $B$ and $P_{\delta}$ is a defect pointed group of $b$,  it is easy to see that Bockstein homomorphisms of group cohomology can be restricted to Bockstein homomorphisms of Linckelmann block cohomology $\beta_{G,b}$. The canonical embedding  \cite[Theorem 5.6 (iii)]{LiTr} from block cohomology to $\HH^*(B)$ is an important ingredient which appears in the main result of \cite{LiVar}. If $\bbbeta_{G,B}$ and $\beta_{G,b}$ are somehow compatible through this canonical map remains a question for further investigations.

In Section \ref{sec4} we study the graded $k$-algebra $\Hc^*(\Delta P, ikGi)$, where $P_{\delta}$ is a defect pointed group of $b$ and $i\in \delta$ is a source idempotent. The $k$-algebra $ikGi$ is called the source algebra and is a $\Delta P$-module, see \cite{TH} for more details. Since the block algebra $B$ is Morita equivalent to its source algebra $ikGi$ an important ingredient for proving the main result of \cite{LiVar} was $\Hc^*(\Delta P, ikGi)$. So, to determine its properties could be useful for future applications. This is the goal of this section where we determine an additive decomposition and a Product Formula for $\Hc^*(\Delta P, ikGi)$ with respect to cohomology of some subgroups in $\Delta P$ with trivial coefficients, see Theorem \ref{thmadditivdecom}. Theorem \ref{thmprodformula}  is one of the main results of this article and can be viewed as a Product Formula of 
the cohomology algebra $\Hc^*(\Delta P, ikGi)$ through the additive decomposition of cohomology graded vector spaces of some finite subgroups in $\Delta P$ with trivial coefficients, from Theorem \ref{thmadditivdecom}, and by using the cup products of restrictions in ordinary group cohomology with trivial coefficients.

In conclusion  this article can be divided in two big parts. In the first part we prove similar results (Theorem \ref{thmBockderiv}, Remark \ref{remdefnbockcohombl}, Theorem \ref{thmbocksource}) for different cohomology algebras which can be summarized in the next theorem:
\begin{theorem} With the above notations  let $\Hc^*$ be one of the cohomology algebras in the set $\{\Hc^*(G,kG), \Hc^*(G,b,P_{\delta}),\Hc^*(\Delta P,ikGi)\}$. Then there is a map $\bbbeta:\Hc^r\rightarrow \Hc^{r+1}$ for any non-negative integer $r$ such that $\bbbeta$ is a differential and satisfies the graded Leibniz rule. In particular $\Hc^*$ is a DG-algebra with respect to $\bbbeta$.
\end{theorem}
 In the second part the main result is the Product Formula, Theorem \ref{thmprodformula}, with its usefulness  demonstrated in Section \ref{sec5} where we define Bockstein homomorphisms for $\Hc^*(\Delta P, ikGi)$ and we show more properties.
  
For results, notations and definitions regarding block algebras we follow \cite{TH}. For results in homological algebra we follow \cite{BenII} and \cite{Ev}.  For conjugation in a group $G$ if $x,a\in G$ then ${}^xa=xax^{-1}$ and $a^x=x^{-1}ax$; similarly for subgroups. If $A$ is any set we denote by $\id_A$ the identity map and by $\mu$ the multiplication map in any $k$-algebra.
\section{Hochschild cohomology of group algebras and Bockstein maps}\label{sec2}

The next lemma is a result which allows us to define  Bockstein maps in Hochschild cohomology of group algebras. The proof is trivial and is left as an exercise.
\begin{lemma}\label{lem21}

\begin{itemize}
\item[i)] There exist a short exact sequence of $\Z G$-modules, where $G$ acts by conjugation
$$\xymatrix{0\ar[r]&\Fp G\ar[r]^{i_G}&\mathbb{Z}/p^2[G]\ar[r]^{\pi_G}&\Fp G\ar[r]&0},$$
and $i_G,\pi_G$ are given by $i_G(\widehat{a}g)=\overline{pa}g,~~\pi_G(\overline{a}g)=\widehat{a}g$ for any $\widehat{a}\in\Fp,\overline{a}\in\mathbb{Z}/p^2$ and $ g\in G$; here $\widehat{a}=a+p\Z\in\Fp$ and $\overline{a}=a+p^2\Z\in\Z/p^2\Z$ for any $a\in \Z$.
\item[ii)]There exist a short exact sequence of $\Z G$-modules, where $G$ acts by conjugation
$$\xymatrix{0\ar[r]&\Z G\ar[r]^{i'_{G}}
&\Z G\ar[r]^{\pi'_{G}}&\Fp G\ar[r]&0},$$
and $i'_{G},\pi'_{G}$ are given by $i'_{G}(ag)=pag,~~\pi'_{G}(ag)=\widehat{a}g$ for any $a\in\Z, g\in G$.
\end{itemize}
\end{lemma}
The long exact sequence theorem applied on the short exact sequences from Lemma \ref{lem21} permit us to consider the next definition.
\begin{definition}\label{defBock}The \emph{Bockstein homomorphism in Hochschild cohomology of group algebras} is the connecting homomorphism obtained from the short exact sequence of Lemma \ref{lem21}, i)
$$\bbbeta_G:\Hc^r(G,\Fp G)\rightarrow\Hc^{r+1}(G,\Fp G)$$
Similarly, for the short exact sequence from Lemma \ref{lem21}, ii) we have the connecting homomorphism
$$\widehat{\bbbeta}_G:\Hc^r(G,\Fp G)\rightarrow\Hc^{r+1}(G,\Z G).$$
\end{definition}
Since we can identify $\Hc^*(G,\Fp G)$ with $\Ext^*_{\Fp G}(\Fp,\Fp G)$ we have that $\bbbeta_G$ is also a homomorphism of $\Fp $-vector spaces.

In order to prove the main result of this section we recall the additive decomposition and the \emph{product formula} from \cite[$\S$4, $\S$5]{SieWith} for $\Hc^*(G,kG)$.  We use the same notations and recall the results in our setting, when  $k$ is a field of characteristic $p$ and $G$ acts by conjugation on $G$; but notice that the results are true for any commutative ring and for any action by automorphisms of a finite group $H$ on the group $G$. Let $g_1=1,g_2,\ldots,g_s$ be a system of representatives of the conjugacy classes of the action of $G$ on $G$, where $s$ is the number of the conjugacy classes. Denote by $G_i:=C_G(g_i),~i\in\{1,\ldots, s\}$ the centralizer subgroup of $g_i$ in $G$. Next, fix $g\in G$ and let $i\in\{1,\ldots, s\}$. We consider the $k(C_G(g))$-module homomorphisms
$$\theta_g:k\rightarrow kG,~~~\theta_g(\lambda)=\lambda g,$$
$$\pi_g:kG\rightarrow k~~~\pi_g(\sum_{a\in G}\lambda_a a)=\lambda_g.$$
If $W$ is a subgroup of $C_G(g)$ then $\theta_g,\pi_g$ induce maps in cohomology
$$\theta_g^*:\Hc^*(W,k)\rightarrow\Hc^*(W,kG),$$
$$\pi_g^*:\Hc^*(W,kG)\rightarrow\Hc^*(W,k).$$
Next, we define
$$\gamma_i:\Hc^*(G_i,k)\rightarrow\Hc^*(G,kG)$$
by $\gamma_i:=\tr_{G_i}^G\circ\theta_{g_i}^*$; for example $\gamma_1:\Hc^*(G,k)\rightarrow\Hc^*(G,kG)$ is $\gamma_1=\theta_{g_1}^*=\theta_1^*$, and is a graded $k$-algebra monomorphism. By \cite[Lemma 4.2]{SieWith} there is an isomorphism of graded $k$-vector spaces, an \emph{additive decomposition}
\begin{equation}\label{eqaddit}\Hc^*(G,kG)\rightarrow\bigoplus_{i=1}^s\Hc^*(G_i,k),~~~
[f]\mapsto(\pi_{g_i}^*(\res^G_{G_i}[f]))_{i}
\end{equation}
with its inverse sending $[f_i]\in\Hc^*(G_i,k)$ to $\gamma_i([f_i]).$ For the Product Formula we fix $i,j\in\{1,\ldots, s\}$ and let $D$ be a set of double coset representatives of $G_i\backslash G/G_j$. For each $x\in D$ there is a unique $l=l(x)$ such that
\begin{equation}\label{eqind}
g_l={}^yg_i{}^{yx}g_j
\end{equation}
for some $y\in G$. Let $W:={}^{yx}G_j\cap{}^y G_i$ and let $[f_i]\in\Hc^*(G_i,k),~[f_j]\in\Hc^*(G_j,k)$, then
\begin{equation}\label{eqproductform}
\gamma_i([f_i])\cup\gamma_j([f_j])=\sum_{x\in D}\gamma_l\left(\tr_W^{G_l}(\res_W^{{}^yG_i}(y^*[f_i])\cup\res_W^{{}^{yx}G_j}((yx)^*[f_j]))\right)
\end{equation}
where $y\in G,x\in D$ and $l$ are like in equation (\ref{eqind}) and $y^*:\Hc^*(G_i,k)\rightarrow\Hc^*({}^yG_i,k)$ is the conjugation map in group cohomology.
\begin{lemma}\label{lem23'} With the above notations the following equality holds
$$\bbbeta_G\circ\gamma_i^r=\gamma_i^{r+1}\circ\beta_{G_i},$$
for any $i\in\{1,\ldots, s\}$.
\end{lemma}
\begin{proof}
  It  is easy to verify the commutativity of the following diagram of short exact sequences of $\Z G_i$-modules
\begin{displaymath}
 \xymatrix{0\ar[r]&\Fp\ar[r]\ar[d]^{\theta_{g_i}}&\mathbb{Z}/p^2\ar[r]\ar[d]^{\theta'_{g_i}} & \Fp\ar[d]^{\theta_{g_i}}\ar[r]&0 \\0\ar[r]&\Fp G\ar[r]^{i_G}&\mathbb{Z}/p^2[G]\ar[r]^{\pi_G}& \Fp G\ar[r]&0
                                            },
\end{displaymath}
where the first short exact sequence of trivial $\Z G_i$-modules induces the Bockstein map in group cohomology and $\theta'_{g_i}$ is the similar map to $\theta_{g_i}$ but with the coefficient ring $\mathbb{Z}/p^2$.
Since $\Hc^*(G_i,-)$ is a cohomolgical $\delta$-functor the following diagram is commutative
\begin{displaymath}
 \xymatrix{\Hc^r(G_i,\Fp)\ar[r]^{\beta_{G_i}}\ar[d]^{\theta_{g_i}^r}&\Hc^{r+1}(G_i,\Fp)\ar[d]^{\theta^{r+1}_{g_i}} \\ \Hc^r(G_i,\Fp G)\ar[r]^{\bbbeta_{G_i}}&\Hc^{r+1}(G_i,\Fp G)
                                            },
\end{displaymath}
where $\bbbeta_{G_i}$ is the connecting homomorphism of the down-side short exact sequence (of $\mathbb{Z}G_i$-modules) of the first diagram of this proof.
Since $\tr_{G_i}^G$ is a morphism of $\delta$-functors we also have the commutativity of the following diagram
\begin{displaymath}
 \xymatrix{\Hc^r(G_i,\Fp G)\ar[r]^{\bbbeta_{G_i}}\ar[d]^{\tr_{G_i}^G}&\Hc^{r+1}(G_i,\Fp G)\ar[d]^{\tr_{G_i}^G} \\ \Hc^r(G,\Fp G)\ar[r]^{\bbbeta_{G}}&\Hc^{r+1}(G,\Fp G)
                                            }.
\end{displaymath}
Gluing the above diagrams we obtain
$$\bbbeta_G\circ\tr_{G_i}^G\circ\theta_{g_i}^r=\tr_{G_i}^G\circ\bbbeta_{G_i}\circ\theta_{g_i}^r=\tr_{G_i}^G\circ\theta_{g_{i}}^{r+1}\circ\beta_{G_i},$$
hence
$$\bbbeta_G\circ\gamma_i^r=\gamma_i^{r+1}\circ\beta_{G_i}.$$
\end{proof}

\begin{theorem}\label{thmBockderiv}With the above notations the following statements are true:
\begin{itemize}
\item[i)] $\bbbeta_G\circ\bbbeta_G=0;$
\item[ii)] The Bockstein of a cup product (with respect to the pairing $\mu$) is given by $$ \bbbeta_G([a]\cup[b])=\bbbeta_G([a])\cup[b]+(-1)^{\mathrm{deg}([a])}[a]\cup\bbbeta_G([b]),$$
    for any $[a],[b]\in\Hc^*(G,\Fp G).$ In particular $\Hc^*(G,\Fp G)$ is a DG-algebra.
\item[iii)] There is a $k$-linear map $\bbbeta_G:\Hc^r(G,kG)\rightarrow\Hc^{r+1}(G,kG)$ such that $\Hc^*(G,kG)$ is also a DG-algebra.
    \end{itemize}
\end{theorem}

\begin{proof}
 \begin{itemize}
  \item[i)] From Lemma \ref{lem21}, Definition \ref{defBock} and the long exact sequence theorem we obtain that $\widehat{\bbbeta}_G\circ(\pi'_G)^*=0$. Next, we consider the commutative diagram of $\Z G$-modules
      \begin{displaymath}
 \xymatrix{0\ar[r]&\Z G\ar[r]^{i'_G}\ar[d]^{\pi'_G}&\Z G\ar[r]^{\pi'_G}\ar[d]^{\pi''_G} & \Fp G\ar[d]^{\id_{\Fp G}}\ar[r]&0 \\0\ar[r]&\Fp G\ar[r]^{i_G}&\mathbb{Z}/p^2[G]\ar[r]^{\pi_G}& \Fp G\ar[r]&0
                                            },
\end{displaymath}
where $\pi''_G$ is the similar map to $\pi'_G$ but with coefficients in $\mathbb{Z}/p^2$ for the codomain.
Since $\Hc^*(G,-)$ is a cohomological $\delta$-functor we obtain the commutative diagram
\begin{displaymath}
 \xymatrix{\Hc^*(G,\Fp G)\ar[r]^{\widehat{\bbbeta}_{G}}\ar[d]^{\id}&\Hc^{*+1}(G,\Z G)\ar[d]^{(\pi'_G)^{*+1}} \\ \Hc^*(G,\Fp G)\ar[r]^{\bbbeta_{G}}&\Hc^{*+1}(G,\Fp G)
                                            }.
\end{displaymath}

 Now, $\bbbeta_G=(\pi'_G)^*\circ\widehat{\beta}_G$ hence
 $$\bbbeta_G\circ \bbbeta_G=(\pi'_G)^{*+1}\circ\widehat{\bbbeta}_G\circ(\pi'_G)^{*+1}\circ\widehat{\bbbeta}_G=0.$$
 \item[ii)]
 The isomorphism (\ref{eqaddit}) for $k$ replaced with $\Fp$ and Lemma \ref{lem23'} give an explicit definition for $\bbbeta_G$, since by (\ref{eqaddit}) we will take $$\{\gamma_i([f_i])\mid i\in\{1,\ldots, s\},~[f_i]\in\Hc^*(G_i,\Fp)\}$$ as a set of generators for $\Hc^*(G,\Fp G)$ and then
 $$\bbbeta_G(\gamma_i([f_i]))=\gamma_i(\beta_{G_i}([f_i])),$$
 for any $i\in\{1,\ldots,s\}$ and any $[f_i]\in\Hc^*(G_i,\Fp).$ Let $[a]=\gamma_i([f_i])$ and $[b]=\gamma_j([f_j])$ be two homogeneous elements in $\Hc^*(G,\Fp G)$, where $i,j\in\{1,\ldots,s\}$ and $[f_i]\in\Hc^*(G_i,\Fp),$$[f_j]\in\Hc^*(G_j,\Fp)$. For shortness, we denote by $m$ the degree $\mathrm{deg}(\res_W^{{}^yG_i}y^*([f_i]))$ in the following equalities. From (\ref{eqproductform}), using the same notations and Lemma \ref{lem23'},  we have
 \begin{align*}
   \bbbeta_G&([a]\cup [b])\\&=\sum_{x\in D}\bbbeta_G\left[\gamma_l\left(\tr_W^{G_l}(\res_W^{{}^yG_i}y^*([f_i])\cup\res_W^{{}^{yx}G_j}(yx)^*([f_j]))\right)\right]\\
   &=\sum_{x\in D}\gamma_l\left[\beta_{G_l}\left(\tr_W^{G_l}(\res_W^{{}^yG_i}y^*([f_i])\cup\res_W^{{}^{yx}G_j}(yx)^*([f_j]))\right)\right]\\
   &=\sum_{x\in D}\gamma_l\left[\tr_W^{G_l}\left(\beta_W(\res_W^{{}^yG_i}y^*([f_i])\cup\res_W^{{}^{yx}G_j}(yx)^*([f_j]))\right)\right]\\
   &=\sum_{x\in D}\gamma_l\left[\tr_W^{G_l}\left(\beta_W(\res_W^{{}^yG_i}y^*([f_i]))\cup\res_W^{{}^{yx}G_j}(yx)^*([f_j])\right.\right.+ \\
   &\left.\left.+(-1)^{m}\res_W^{{}^yG_i}y^*([f_i])\cup\beta_W(\res_W^{{}^{yx}G_j}(yx)^*([f_j]))\right)\right]
   \\
   &=\sum_{x\in D}\gamma_l\left[\tr_W^{G_l}\left(\res_W^{{}^yG_i}(y^*(\beta_{G_i}[f_i]))\cup\res_W^{{}^{yx}G_j}(yx)^*([f_j])\right.\right.+ \\
   &\left.\left.+(-1)^{m}\res_W^{{}^yG_i}y^*([f_i])\cup\res_W^{{}^{yx}G_j}((yx)^*(\beta_{G_j}[f_j]))\right)\right]
   \\
   &=\sum_{x\in D}\gamma_l\left[\tr_W^{G_l}\left(\res_W^{{}^yG_i}(y^*(\beta_{G_i}[f_i]))\cup\res_W^{{}^{yx}G_j}(yx)^*([f_j])\right)\right]+ \\
   &+(-1)^{m}\sum_{x\in D}\gamma_l\left[\tr_W^{G_l}\left(\res_W^{{}^yG_i}y^*([f_i])\cup\res_W^{{}^{yx}G_j}((yx)^*(\beta_{G_j}[f_j]))\right)\right]\\
   &=\gamma_i(\beta_{G_i}[f_i])\cup\gamma_j([f_j])+(-1)^{\mathrm{deg}([f_i])}\gamma_i([f_i])\cup\gamma_j(\beta_{G_j}([f_j]))\\
   &=\bbbeta_G([a])\cup[b]+(-1)^{\mathrm{deg}([a])}[a]\cup\bbbeta_G([b]),
 \end{align*}
 where the fourth equality is true since $\beta_W$ is a graded derivation; some of the other equalities are true since restriction, transfer and conjugation map are homomorphisms of $\delta$-functors.
 \item[iii)]
 The field $\Fp$ is a subfield of $k$, hence $k$ is an $\Fp$-algebra. Then there is an isomorphism of graded $k$-algebras
$$\Hc^*(G,\Fp )\otimes_{\Fp}k\cong \Hc^*(G,k),$$
see \cite[p. 30]{Ev} for an explicit description of this isomorphism. This isomorphism induces a $k$-linear map $\beta_G:\Hc^r(G,k)\rightarrow\Hc^{r+1}(G,k),$
  for which by abuse of notation we use the same notation as in the case of $\Hc^*(G,\Fp)$ and is obtained as the tensor product over $\Fp$ of the above  Bockstein map  with $\id_k$. Obviously $\beta_G$ is a differential and a graded derivation. The isomorphism (\ref{eqaddit}) for $k$ assure us that we may give the same explicit definition given in the proof of statement ii) for a Bockstein homomorphism in $\Hc^*(G,kG)$
$$\bbbeta_G:\Hc^r(G,kG)\rightarrow\Hc^{r+1}(G,kG)$$
$$\bbbeta_G(\gamma_i([f_i]))=\gamma_i(\beta_{G_i}([f_i])),$$
 for any $i\in\{1,\ldots,s\}$ and any $[f_i]\in\Hc^*(G_i,k).$ The same computations as in the proof of statement ii) assure us that $\bbbeta_G$ is a differential and a graded derivation.
 \end{itemize}
\end{proof}
\begin{remark} We see that in the proof of statement iii) of the above theorem we have used the extension of the Bockstein map from $\Hc^*(G,\Fp)$ to a $k$-linear map, which we will call still the Bockstein map of $\Hc^*(G,k)$, via the isomorphism of graded $k$-algebras $$\Hc^*(G,\Fp )\otimes_{\Fp}k\cong \Hc^*(G,k),$$ as in \cite[p.133]{BenII} or in \cite[p.30]{Ev}. But in a revised version of \cite{BenII} the author noticed that this extension is not a good thing to do since it caused problems for \cite[Section 5.8]{BenII}. Instead he uses a semi-linear extension via the Frobenius map of the Bockstein map from cohomology with coefficients in $\Fp$ to the cohomology with coefficients in $k$. It is possible that many of the above and next results to remain true by considering this map. However for our purposes the considered extension from \ref{thmBockderiv}, iii) is good and it will be used for the rest of this paper. Also notice that from Section \ref{sec3} we will work, for simplicity, with an algebraically closed field $k$ of  characteristic $p$ and in this case these two maps are quite similar, see \cite[Exercise 3.4.2]{Ev}.
\end{remark}
\begin{remark}\label{remexplbockhochsch}
  From Definition \ref{defBock} we notice that if $k=\Fp$, Bockstein homomorphisms for Hochschild cohomology of group algebras are defined as the connecting homomorphisms in some long exact sequences, however in general, for any field $k$ of prime characteristic, from the proof of Theorem \ref{thmBockderiv}, iii) we see that we can define Bockstein homomorphism $$\bbbeta_G:\Hc^r(G,kG)\rightarrow\Hc^{r+1}(G,kG)$$
$$\bbbeta_G(\gamma_i([f_i]))=\gamma_i(\beta_{G_i}([f_i])),$$
 for any $i\in\{1,\ldots,s\}$ and any $[f_i]\in\Hc^*(G_i,k)$, only using Siegel-Witherspoon additive decomposition. We proved its properties using Product Formula and properties of Bockstein homomorphisms for group cohomology with coefficients in $k$.
\end{remark}
\section{Hochschild cohomology of block algebras and Bockstein homomorphisms}\label{sec3}
In this section let $k$ be an algebraically closed field of characteristic $p$ and $\mathcal{B}=\{b| b\in Z(kG)\}$  be the finite set of block idempotents of the group algebra $kG$. Each block idempotent determines a block algebra $kGb$, which is an indecomposable factor of $kG$ as $k$-algebra.  We have a well-known decomposition
$$kG=\bigoplus_{b\in\mathcal{B}}kGb$$ as $k[G\times G]$-modules, where the action of $G\times G$ on $kG$ is given by $(g,h)a=gah^{-1}$ for any $g,h\in G,a\in kG$. It follows that $$kG=\bigoplus_{b\in\mathcal{B}}kGb$$
as $kG$-modules, where $G$ acts by conjugation, hence
$$\Hc^*(G,kG)\cong \bigoplus_{b\in\mathcal{B}}\Hc^*(G,kGb)$$ as graded $k$-algebras. In particular, for the fixed non-negative integer $r$ we have
$$\Hc^r(G,kG)\cong \bigoplus_{b\in\mathcal{B}}\Hc^r(G,kGb)$$
as $k$-vector spaces. Next, fix a block algebra from $\{kGb| b\in\mathcal{B}\}$, which we denote by $B=kGb$, where $b$ is its fixed block idempotent. Again $B$ is a $kG$-module by conjugation, since we identify $kG\cong k\Delta G$, where $\Delta G\leq G\times G$ is the diagonal subgroup. We define the following $kG$-module homomorphisms
$$\mathbb{i}_{G,B}:B\hookrightarrow kG$$
$$\mathbb{p}_{G,B}:kG\rightarrow B,$$
where $\mathbb{i}_{G,B}$ is the  canonical inclusion map and $\mathbb{p}_{G,B}$ is the canonical projection, obtained by multiplication with $b$. Clearly $\mathbb{p}_{G,B}\circ \mathbb{i}_{G,B}=\id_B$ and these maps induce the following maps in cohomology:

$$\mathbb{i}_{G,B}^*:\Hc^*(G,B)\rightarrow\Hc^*(G,kG)$$
$$\mathbb{p}_{G,B}^*:\Hc^*(G,kG)\rightarrow\Hc^*(G,B)$$
such that $\mathbb{p}_{G,B}^*\circ \mathbb{i}_{G,B}^*=\id_{\Hc^*(G,B)}.$
Moreover, since $\mathbb{i}_{G,B},\mathbb{p}_{G,B}$ are homomorphisms of $k$-algebras as well, we have that $\mathbb{i}_{G,B}^*,\mathbb{p}_{G,B}^*$ are homomorphisms of graded commutative $k$-algebras, where $\Hc^*(G,B)$ is again an associative graded algebra with cup product with respect to the pairing $\mu$; here $\mu$ is the multiplication map in $B$.
\begin{definition}\label{defnBockHHbloc}
The \emph{Bockstein homomorphism in Hochschild cohomology of a block algebra} $B$ is the map
$$\bbbeta_{G,B}:\Hc^r(G,B)\rightarrow\Hc^{r+1}(G,B)$$
given by
$$\bbbeta_{G,B}([f])=\mathbb{p}_{G,B}^{r+1}(\bbbeta_G(\mathbb{i}_{G,B}^r([f]))$$
for any  $[f]\in\Hc^r(G,B)$.
\end{definition}
With the notations from Section \ref{sec2} for each $i\in\{1,\ldots, s\}$ we define the homomorphism of graded $k$-vector spaces
 $$\gamma_{i,G,B}:=\mathbb{p}_{G,B}^*\circ \gamma_i$$
  $$\gamma_{i,G,B}:\Hc^*(G_i,k)\rightarrow\Hc^*(G,B)$$

  and $$\Hc_B^*(G_i,k):=\Ima(\pi_{g_i}^*\circ\res^G_{G_i}\circ \mathbb{i}_{G,B}^*)$$ as a graded $k$-subspace of $\Hc^*(G_i,k)$. Since $\pi_{g_i}$ is only additive in general, $\Hc^*_B(G_i,k)$ is not necessarily a subring of $\Hc^*(G_i,k)$.

\begin{proposition}\label{prop31addithochbloc} With the above notations the map
$$\Phi:\Hc^*(G,B)\rightarrow\bigoplus_{i=1}^s\Hc_B^*(G_i,k)$$
$$\Phi([f])=\pi_{g_i}^*(\res^G_{G_i}(\mathbb{i}_{G,B}^*[f])),~~~[f]\in\Hc^*(G,B)$$
is an isomorphism of graded $k$-vectors spaces with its inverse sending an element $[f_i]\in\Hc_B^*(G_i,k)$ to $\gamma_{i,G,B}([f_i]).$
\end{proposition}
\begin{proof}
  Clearly $\Phi$ is well-defined and surjective. Since
  $$\gamma_{i,G,B}(\Phi([f]))=\gamma_{i,G,B}[\pi_{g_i}^*(\res^G_{G_i}(\mathbb{i}_{G,B}^*([f])))]$$
  $$=\mathbb{p}_{G,B}^*\left((\gamma_i\circ\pi_{g_i}^*\circ\res^G_{G_i})(\mathbb{i}_{G,B}^*([f]))\right)$$
  $$=(\mathbb{p}_{G,B}^*\circ \mathbb{i}_{G,B}^*)([f])=[f],$$
  for any $[f]\in\Hc^*(G,B)$ we obtain the conclusion.
\end{proof}
From Proposition \ref{prop31addithochbloc} we obtain that the set
$$\{\gamma_{i,G,B}([f_i])|i\in\{1,\ldots,s\},[f_i]\in\Hc^r_B(G_i,k)\}$$ is a $k$-generating set for $\Hc^r(G,B)$. This allows us to state the next remark.

\begin{remark}
It is obvious that by Proposition \ref{prop31addithochbloc} an equivalent formulation for $\bbbeta_{G,B}$ is $$\bbbeta_{G,B}(\gamma_{i,G,B}[f_i])=\gamma_{i,G,B}(\beta_{G_i}([f_i]))$$
for any $i\in\{1,\ldots,s\}$ and any $[f_i]\in\Hc^r_B(G_i,k)$. Equivalently
$$\bbbeta_{G,B}(\gamma_{i,G,B}[f_i])=\mathbb{p}_{G,B}^{r+1}(\bbbeta_G(\gamma_i([f_i])))$$
for any $i\in\{1,\ldots,s\}$ and $[f_i]\in\Hc^r_B(G_i,k)$.
Theorem \ref{thmBockderiv} and the fact that $\mathbb{p}_{G,B}^*,\mathbb{i}_{G,B}^*$ are homomorphisms of graded $k$-algebras assure us that $\bbbeta_{G,B}$ is  also a graded derivation.
\end{remark}

We end this section with a remark regarding Bockstein homomorphisms for Linckelmann cohomology of block algebras. Let $P_{\delta}$ be a defect pointed group in $G_{\{b\}}$ and $(P,e_P)$ be an associated maximal $b$-Brauer pair. It is well known that for any $Q\leq P$ there is a unique block $e_Q$ in $kC_G(Q)$ such that $(Q,e_Q)$ is a subpair of $(P,e_P)$ and the set of $b$-Brauer pairs is a $G$-set. Block cohomology was introduced by Linckelmann in \cite{LiTr} and is denoted by $\Hc^*(G,b,P_{\delta})$. Recall that $\Hc^*(G,b,P_{\delta})$ is formed by all elements $[\zeta]\in\Hc^*(P,k)$ such that
$$\res^P_Q([\zeta])=(g^*\circ\res^P_Q)([\zeta]),$$
for any $g\in N_G(Q,e_Q)$. Moreover as a graded $k$-algebra, we have the equality
$$\Hc^*(G,b,P_{\delta})=\bigoplus_{r\geq 0}\Hc^r(G,b,P_{\delta}).$$
\begin{remark}\label{remdefnbockcohombl} Restriction and conjugation are homomorphisms of cohomological $\delta$-functors, hence for $k=\Fp$ we know that $\res^P_Q$ and $g^*$ commutes with  $\beta_P$, respectively $\beta_Q$ for any $Q\leq P$ and $g\in N_G(Q)$. But also for any field $k$ of characteristic $p$ it is easy to prove that these properties hold, since $\res^P_Q\otimes_{\Fp}\id_k$ and $g^*\otimes_{\Fp}\id_k$ can be identified with $\res^P_Q$ respectively $g^*$. From the proof of Theorem \ref{thmBockderiv}, iii) recall that $\beta_P:\Hc^r(P,k)\rightarrow \Hc^{r+1}(P,k)$ is the tensor product over $\Fp$ of the  Bockstein map for $\Fp$  with $\id_k$, it follows that there is a well-defined homomorphism
$$\beta_{G,b}:\Hc^r(G,b,P_{\delta})\rightarrow\Hc^{r+1}(G,b,P_{\delta}),$$
which is the restriction of $\beta_P$ to $\Hc^*(G,b,P_{\delta})$. Since $\beta_P$ is a differential and a graded derivation we have that $\beta_{G,b}$ is also a differential and a graded derivation.
\end{remark}
\section{A Product Formula for cohomology of defect groups with coefficients in source algebras}\label{sec4}
Before obtaining the main results of this section let us introduce some notations and collect some properties in the first lemma of this section, without proof, for some subgroups in $\Delta P$. For any subgroup $P\leq G$ let $x\in G,(u,v)\in P\times P$. We denote the subgroups
$$Q_x:=\{(a,a^x)|a\in P\cap{}^xP\}$$
and $Q_{u,v,x}:=\Delta P\cap {}^{(u,v)}Q_x.$ If $w\in \Delta P$, by abuse of notation, we sometimes denote $w:=(w,w)$.
\begin{lemma}\label{lemproprQuvx} Let $x\in G,w\in\Delta P$ and $(u,v)\in P\times P$. The following statements hold.
\begin{itemize}
\item[i)] $Q_{u,v,x}=\Delta P\cap{}^{(ux,v)}\Delta(P^x\cap P)=\Delta P\cap {}^{(u,vx^{-1})}\Delta(P\cap {}^xP);$
\item[ii)] Let $(a,a^x)\in Q_x$ where $a\in P\cap{}^x P$. Then $Q_{wu,wv,x}={}^wQ_{u,v,x}$ and $Q_{ua,va^x,x}=Q_{u,v,x};$
\item[iii)] If $w\in Q_{u,v,x}$ then $w\in C_P(uxv^{-1})$.
\end{itemize}
\end{lemma}
From now, for the rest of this section, we consider as in the previous section a block algebra $B=kGb$ with defect pointed group $P_{\delta}$. Let $i\in\delta $ be a source idempotent, then $ikGi$ is a $P$-interior algebra, called the \emph{source algebra} of the block $b$. In particular $ikGi$ is also a $kP-kP$-bimodule, hence a left $k[P\times P]$-module by $(u,v)(imi)=iumv^{-1}i$, for any $u,v\in P,m\in kG$. By \cite[Theorem 44.3]{TH}  we have a decomposition
$$ikGi\cong \bigoplus_{x\in Y_{G,b,P_{\delta}}}k[PxP]$$
as indecomposable $k[P\times P]$-modules, where $Y_{G,b,P_{\delta}}$ is a subset of a system of representatives of double cosets of $P$ in $G$. We denote this isomorphism by
$$s:ikGi\rightarrow \bigoplus_{x\in Y_{G,b,P_{\delta}}}k[PxP],$$
and its inverse by $s^{-1}$.
For $x\in Y_{G,b,P_{\delta}}$, any element $(u,v)\in P\times P$ and $W\leq Q_{u,v,x}$,  by Lemma \ref{lemproprQuvx}, iii) we define the following homomorphisms of $kW$-modules

\begin{equation}
\begin{split}
&\theta_{u,v,x}:k\rightarrow ikGi \\
\theta_{u,v,x}(\alpha)&=\alpha s^{-1}(uxv^{-1}),~~~\alpha\in k
\end{split}
\end{equation}
\begin{equation}
\begin{split}
  &\pi_{u,v,x}:ikGi\rightarrow k\\
  \pi_{u,v,x}(m)=\delta_{u,v,x}&(\pi_x(s(m))),~~~m\in ikGi
  \end{split}
\end{equation}

where $\delta_{u,v,x}:k[PxP]\rightarrow k$ is given by
$$\delta_{u,v,x}(\mathfrak{u}x\mathfrak{v})=\left\{ \begin{array}{ll}
1,\quad \text{if}~(\mathfrak{u},\mathfrak{v}^{-1})\in Q_{u,v,x}(u,v)Q_x\\
0,\quad \text{if}~(\mathfrak{u},\mathfrak{v}^{-1})\notin Q_{u,v,x}(u,v)Q_x
\end{array} \right.$$ and $\pi_x:\bigoplus_{x\in Y_{G,b,P_{\delta}}}k[PxP]\rightarrow k[PxP]$ is the $k[P\times P]$-module canonical projection.
\begin{lemma}\label{lemproprtetapi} Let $x\in Y_{G,b,P_{\delta}},(u,v)\in P\times P,w\in Q_{u,v,x}$ and $(a,a^x)\in Q_x$. The following statements hold.
\begin{itemize}
  \item[i)] $\theta_{wua,wva^x,x}=\theta_{u,v,x};$
  \item[ii)]$\pi_{wua,wva^x,x}=\pi_{u,v,x}.$
\end{itemize}
\end{lemma}
\begin{proof}
  Statement ii) is left for the reader. For statement i) let us take an element $\alpha\in k$. It follows
  $$\theta_{wua,wva^x,x}(\alpha)=\alpha s^{-1}(wuax(a^{-1})^xv^{-1}w^{-1})=\alpha s^{-1}(wuaxx^{-1}a^{-1}xv^{-1}w^{-1})$$
  $$=\alpha s^{-1}(uxv^{-1}),$$
  where the last equality is true by Lemma \ref{lemproprQuvx}, iii).
\end{proof}

We emphasize that in this section, in contrast to Sections \ref{sec2} and \ref{sec3}, we prefer to use the more suggestive notation for cohomology of a group with coefficients in the group algebra or in the source algebra $\Hc^*(\Delta P,ikGi)$, opposed to  $\Hc^*(P,ikGi)$ where $P$ acts by conjugation on $ikGi$.
With the above notations $\pi_{u,v,x},\theta_{u,v,x}$ induce maps in cohomology
\begin{equation}\label{eqtetapi*}
\begin{split}
\theta_{u,v,x}^*:\Hc^*(W,k)\rightarrow\Hc^*(W,ikGi)\\
\pi_{u,v,x}^*:\Hc^*(W,ikGi)\rightarrow\Hc^*(W,k)
\end{split}
\end{equation}
Moreover we define the maps in cohomology
\begin{equation}\label{eqgamma}
\begin{split}
  \gamma_{u,v,x}&:\Hc^*(Q_{u,v,x},k)\rightarrow\Hc^*(\Delta P,ikGi)\\
  &\gamma_{u,v,x}:=\tr_{Q_{u,v,x}}^{\Delta P}\circ\theta_{u,v,x}^*
\end{split}
\end{equation}
here $\theta_{u,v,x}^*$ is obtained for $W=Q_{u,v,x}$.
From Lemma \ref{lemproprtetapi}  and Lemma \ref{lemproprQuvx}  we immediately collect some properties of the above maps (\ref{eqtetapi*}), (\ref{eqgamma}) in the next proposition.
\begin{proposition}\label{propproprtetaapigamma}
Let $x\in Y_{G,b,P_{\delta}},(u,v)\in P\times P, W\leq Q_{u,v,x},w\in Q_{u,v,x}$ and $(a,a^x)\in Q_x$. The following statements hold.
\begin{itemize}
  \item[i)]$\theta_{wua,wva^x,x}^*=\theta_{u,v,x}^*$ and $\gamma_{wua,wva^x,x}^*=\gamma_{u,v,x}^*;$
  \item[ii)]$\pi_{wua,wva^x,x}^*=\pi_{u,v,x}^*.$
\end{itemize}
\end{proposition}
For $x\in Y_{G,b,P_{\delta}}$ by \cite[Lemma 44.1, c)]{TH} there is an isomorphism of $k[P\times P]$-modules
$$k[PxP]\cong \Ind_{Q_x}^{P\times P}k,$$ hence it  follows the isomorphism
$$ikGi\cong \bigoplus_{x\in Y_{G,b,P_{\delta}}}\Res^{P\times P}_{\Delta P}\Ind^{P\times P}_{Q_x}k,$$ as $k\Delta P$-modules. For each $x\in Y_{G,b,P_{\delta}}$ we choose once and for all a system of representatives, which we denote by $[Q_x]$, of double cosets $\Delta P\backslash P\times P/Q_x$. By Mackey decomposition we get
$$\Res^{P\times P}_{\Delta P}\Ind^{P\times P}_{Q_x}k\cong \bigoplus_{(u_x,v_x)\in[Q_x]}\Ind_{\Delta P\cap{}^{(u_x,v_x)}Q_x}^{\Delta P}\Res^{{}^{(u_x,v_x)}Q_x}_{\Delta P\cap{}^{(u_x,v_x)}Q_x}{}^{(u_x,v_x)}k.$$
For shortness, we will identify from now ${}^{(u_x,v_x)}k$ with $k$. We can do this, since $k$ is the trivial $k[\Delta P\cap{}^{(u_x,v_x)}Q_x]$-module; moreover we will use the notations $$Q_{u_x,v_x}=Q_{u_x,v_x,x},\theta_{u_x,v_x}=\theta_{u_x,v_x,x},\pi_{u_x,v_x}=\pi_{u_x,v_x,x},~\text{etc}. $$ Thus we have the decomposition
\begin{equation}\label{eq8}
ikGi \cong\bigoplus_{\substack{x\in Y_{G,b,P_{\delta}}\\(u_x,v_x)\in[Q_x]}}\Ind_{Q_{u_x,v_x}}^{\Delta P}k
\end{equation}
as $k\Delta P$-modules.
In the first main result of this section we obtain an additive decomposition for the graded $k$-vector space $\Hc^*(\Delta P,ikGi)$.
\begin{theorem}\label{thmadditivdecom} With the above notations the map
$$\Hc^*(\Delta P,ikGi)\rightarrow\bigoplus_{\substack{x\in Y_{G,b,P_{\delta}}\\(u_x,v_x)\in[Q_x]}}\Hc^*(Q_{u_x,v_x},k)$$
$$[\zeta]\mapsto\sum_{\substack{x\in Y_{G,b,P_{\delta}}\\(u_x,v_x)\in[Q_x]}}(\pi_{u_x,v_x}^*\circ \res^{\Delta P}_{Q_{u_x,v_x}})([\zeta])$$
is an isomorphism of graded $k$-vector spaces with is inverse sending an element $[f]\in\Hc^*(Q_{u_x,v_x},k)$ to $\gamma_{u_x,v_x}([f])$.

\end{theorem}
\begin{proof}
We have the isomorphisms of $k\Delta P$-modules
$$ikGi \cong \bigoplus_{x\in Y_{G,b,P_{\delta}}}k[PxP]\cong \bigoplus_{x\in Y_{G,b,P_{\delta}}}\Res^{P\times P}_{\Delta P}\Ind^{P\times P}_{Q_x}k,$$
where the first isomorphism is $s$ and the second isomorphism is given by
$$\mathfrak{u}x\mathfrak{v}\mapsto(\mathfrak{u},\mathfrak{v}^{-1})\otimes_{kQ_x}1$$ with its inverse
$$(\mathfrak{u},\mathfrak{v})\otimes_{kQ_x}1\mapsto \mathfrak{u}x\mathfrak{v}^{-1}$$
for any $\mathfrak{u},\mathfrak{v}\in P$ and $x\in Y_{G,b,P_{\delta}}$. We obtain an isomorphism in cohomology
$$\Hc^*(\Delta P,ikGi)\cong \bigoplus_{x\in Y_{G,b,P_{\delta}}}\Hc^*(\Delta P,\Res^{P\times P}_{\Delta P}\Ind^{P\times P}_{Q_x}k).$$
Now for each $x\in Y_{G,b,P_{\delta}}$ as we have seen in (\ref{eq8}) that Mackey decomposition tells us
$$\Res^{P\times P}_{\Delta P}\Ind^{P\times P}_{Q_x}k\cong \bigoplus_{(u_x,v_x)\in[Q_x]}\Ind_{Q_{u_x,v_x}}^{\Delta P}k.$$
Explicitly this isomorphism is obtained from the decomposition of $P\times P$ in double cosets, that is for any $(\mathfrak{u},\mathfrak{v}^{-1})\in P\times P$ there is a unique representative $(u_x,v_x)\in[Q_x]$ such that
$$(\mathfrak{u},\mathfrak{v}^{-1})=(\mathfrak{u}',\mathfrak{u}')(u_x,v_x)(a,a^x)$$
for some $\mathfrak{u}'\in\Delta P$ and $(a,a^x)\in Q_x$. The isomorphism is given by
$$(\mathfrak{u},\mathfrak{v}^{-1})\otimes_{kQ_x}1\mapsto (\mathfrak{u}',\mathfrak{u}')\otimes_{k[Q_{u_x,v_x}]}1.$$
This induces an isomorphism in cohomology $$\Hc^*(\Delta P,\Res^{P\times P}_{\Delta P}\Ind^{P\times P}_{Q_x}k)\cong \bigoplus_{(u_x,v_x)\in[Q_x]}\Hc^*(\Delta P,\Ind_{Q_{u_x,v_x}}^{\Delta P}k).$$
Composing it with the isomorphism of \cite[Lemma 4.1]{SieWith} one obtains the required isomorphism. Recall that in our case the isomorphism from \cite[Lemma 4.1]{SieWith} is obtained by composing $\res^{\Delta P}_{Q_{u_x,v_x}}$  with the map in cohomology  induced by
$$\Ind_{Q_{u_x,v_x}}^{\Delta P}k\rightarrow k,$$
which, with our notations (for  $\mathfrak{u}'\in P$), is given by
$$(\mathfrak{u}',\mathfrak{u}')\otimes_{k[Q_{u_x,v_x}]} 1\mapsto\left\{ \begin{array}{ll}
1,\quad \text{if}~(\mathfrak{u}',\mathfrak{u}')\in Q_{u_x,v_x}\\
0,\quad\text{if}~(\mathfrak{u}',\mathfrak{u}')\notin Q_{u_x,v_x}
\end{array} \right..$$
\end{proof}
\begin{remark}\label{remgammau1}
From \cite[Lemma 44.6]{TH} we know that $ikGi$ always has a direct summand isomorphic to $kP$ as $k[P\times P]$-module. It follows that we can choose $Y_{G,b,P_{\delta}}$ such that there is $x\in PC_G(P)$ in $Y_{G,b,P_{\delta}}$ (we take $x=1$), and in this case $Q_1=\Delta P$ with
$$kP\cong \Ind_{\Delta P}^{P\times P}k.$$
Moreover we obtain
$$\Res^{P\times P}_{\Delta P}\Ind^{P\times P}_{\Delta P}k\cong\bigoplus_{(u_1,v_1)\in[Q_1]}\Ind_{Q_{u_1,v_1}}^{\Delta P}k.$$
In the set $[Q_1]$ there is a representative of the double coset $\Delta P$, for example we take $(u_1,v_1)=(1,1)\in N_{P\times P}(\Delta P)$. It follows that $Q_{1,1}=\Delta P$ and the map
$$\theta_{1,1}^*:\Hc^*(\Delta P,k)\rightarrow\Hc^*(\Delta P,ikGi)$$
is actually the homomorphism $i_P^{ikGi}$ constructed and studied for any interior $P$-algebra by Linckelmann in \cite[Sections 2,3]{LiVariso} and applied for the source algebra in \cite[Section 5]{LiVariso}.
\end{remark}
Next we will give a way to describe the cup product in $\Hc^*(\Delta P,ikGi)$ in terms of direct sums decomposition described in Theorem \ref{thmadditivdecom}. We accomplish this as follows. Fix $x,y\in Y_{G,b,P_{\delta}}$ and $(u_x,v_x)\in[Q_x], (u_y,v_y)\in[Q_y]$ as in the statements after Proposition \ref{propproprtetaapigamma}, where an  element from $[Q_x]$ was denoted $(u_x,v_x)$ and similarly for $[Q_y]$. For each $w\in\Delta P$ there is a unique subset $$Z=\{z|z=z(x,y,u_x,v_x,wu_y,wv_y)\}$$ in $Y_{G,b,P_{\delta}}$, a set $A_z=\{(u'_z,v'_z)|(u'_z,v'_z)\in P\times P\}$ and $\alpha_z\in k,z\in Z$ such that
\begin{equation}\label{eqs-1s-1}
  s^{-1}(u_xxv_x^{-1})s^{-1}(wu_yyv_y^{-1}w^{-1})=\sum_{\substack{z\in Z\\(u'_z,v'_z)\in A_z}}\alpha_zs^{-1}(u'_zz(v'_z)^{-1})
\end{equation}
For two different $z$'s, for example $z\neq z', z,z'\in Z$, we can have the same pairs $(u'_z,v'_z)=(u'_{z'},v'_{z'})\in A_z\cap A_{z'}$, or not. If the  elements $w$ are from a chosen set of representatives of $Q_{u_x,v_x}\backslash\Delta P/Q_{u_y,v_y}$ then any element $z\in Z$ is independent of the choice of representatives of double cosets $Q_{u_x,v_x}\backslash\Delta P/Q_{u_y,v_y}$; this is a consequence of Lemma \ref{lemproprQuvx}, iii).

For each $w\in[Q_{u_x,v_x}\backslash\Delta P/Q_{u_y,v_y}]$ with the corresponding attached unique set $Z$, and for an element $(u'_z,v'_z)\in A_z$ (where $z\in Z$) such that equation (\ref{eqs-1s-1}) holds, there is a unique element $(u_z,v_z)\in [Q_z]$ such that
\begin{equation}\label{equ'zv'z}
(ru'_z,rv'_z)=(u_za,v_za^z)
\end{equation}
for some $r\in \Delta P$ and $(a,a^z)\in Q_z$.
The set of all $r\in \Delta P$ which satisfies (\ref{equ'zv'z}) is also a double coset, that is $Q_{u_z,v_z}\backslash\Delta P/W'$; for shortness, we denoted (and we will denote) the subgroup $Q_{u_x,v_x}\cap Q_{wu_y,wv_y}$ by $W'$ and the set $[Q_{u_x,v_x}\backslash\Delta P/Q_{u_y,v_y}]$, of representatives of double cosets, by $D$ .

\begin{theorem}\label{thmprodformula} With the above notations let $[f]\in \Hc^*(Q_{u_x,v_x},k)$ and $[g]\in\Hc^*(Q_{u_y,v_y},k)$. Then
$$\gamma_{u_x,v_x}([f])\cup\gamma_{u_y,v_y}([g])$$
$$=\sum_{\substack{w\in D\\z\in Z\\(u_z,v_z)}}\alpha_z\gamma_{u_z,v_z}\left(\tr_{Q_{u_z,v_z}\cap{}^r W'}^{Q_{u_z,v_z}}(\res^{{}^rQ_{u_x,v_x}}_{Q_{u_z,v_z}\cap{}^rW'}r^*([f])\cup\res^{{}^{rw}Q_{u_y,v_y}}_{Q_{u_z,v_z}\cap{}^rW'}(rw)^*([g]))\right)$$
where $D$ is a set of representatives of double cosets $Q_{u_x,v_x}\backslash\Delta P/Q_{u_y,v_y}$, the elements $z\in Z,$ $(u_z,v_z)=(u_z(w),v_z(w))$ and $r=r(w), r\in \Delta P$ are chosen to satisfy equations (\ref{eqs-1s-1}), (\ref{equ'zv'z}) and $W':=Q_{u_x,v_x}\cap {}^wQ_{u_y,v_y}$.
\end{theorem}
To prove the theorem, we first deduce some lemmas where we collect more properties for the maps $\theta^*_{u,v,x}$ and $\pi^*_{u,v,x}$.
\begin{lemma}\label{lemproprthetcircpi}
  Let $x,y\in Y_{G,b,P_{\delta}},(u,v),(u',v')\in P\times P$, the subgroups $W_1\leq W\leq Q_{u,v,x}$ and $w\in\Delta P$. The following statements hold.
\begin{itemize}
  \item[i)] $w^*\circ\theta_{u,v,x}^*=\theta^*_{wu,wv,x}\circ w^*;$
  \item[ii)] The maps $\theta^*_{u,v,x}$ and $\pi^*_{u,v,x}$ commute with $\res^W_{W_1}$ respectively $\tr^W_{W_1}$.
  \item[iii)] We have that $\pi^*_{u,v,x}\circ\theta^*_{u',v',y}$ is the identity map if $x=y$ and $(u,v)\in Q_{u,v,x}(u',v')Q_x$, and is the zero map otherwise. In particular if $(u,v)=(u_x,v_x)\in[Q_x]$ and $(u',v')=(u'_y,v'_y)\in[Q_y]$ then $\pi^*_{u_x,v_x}\circ \theta^*_{u'_y,v'_y}$ is the identity map if $x=y$ and $(u_x,v_x)=(u'_x,v'_x)$, and is the zero map otherwise.
\end{itemize}
\end{lemma}
\begin{proof}
  \begin{itemize}
\item[i)] Recall that by Lemma \ref{lemproprQuvx}, ii) we know that ${}^wQ_{u,v,x}=Q_{wu,wv,x}$ and that $w^*:\Hc^*(W,k)\rightarrow\Hc^*({}^wW,k)$ is the conjugation map in cohomology. Let $[f]\in \Hc^*(W,k)$ and $\mathfrak{a}\in\mathcal{P}_*$ where $\mathcal{P}_*\rightarrow k$ is a $k\Delta P$-projective resolution of $k$ as trivial $k\Delta P$-module. We have that $(w^*\circ\theta^*_{u,v,x})[f]=[w^*(\theta_{u,v,x}\circ f)]$ with
    $$w^*(\theta_{u,v,x}\circ f)(\mathfrak{a})=w(\theta_{u,v,x}\circ f)(w^{-1}\mathfrak{a})=w(f(w^{-1}\mathfrak{a})s^{-1}(uxv^{-1}))$$$$=f(w^{-1}\mathfrak{a})s^{-1}(wuxv^{-1}w^{-1}).$$
    Also we have that
    $(\theta^*_{wu,wv,x}\circ w^*)[f]=[\theta_{wu,wv,x}\circ w^*(f)]$ with
    $$\theta_{wu,wv,x}(w^*(f)(\mathfrak{a}))=w^*(f)(\mathfrak{a})s^{-1}(wuxv^{-1}w^{-1})$$$$=(wf(w^{-1}\mathfrak{a})s^{-1}(wuxv^{-1}w^{-1}))=f(w^{-1}\mathfrak{a})s^{-1}(wuxv^{-1}w^{-1}).$$
    Notice that this statement is true for any $w\in\Delta P$.
    \item[ii)] We verify only that $\tr_{W_1}^W\circ\theta^*_{u,v,x}=\theta^*_{u,v,x}\circ \tr_{W_1}^W$, the other statements are similar and left for the reader. Let $[f]\in \Hc^*(W,k)$. We obtain
    $$\tr_{W_1}^W(\theta_{u,v,x}(f))(\mathfrak{a})=\sum_{w\in W/W_1}w\theta_{u,v,x}(f(w^{-1}\mathfrak{a}))$$$$=\sum_{w\in W/W_1}f(w^{-1}\mathfrak{a})s^{-1}(wuxv^{-1}w^{-1})=\sum_{w\in W/W_1} f(w^{-1}\mathfrak{a})s^{-1}(uxv^{-1}),$$
    where the last equality holds since $w\in Q_{u,v,x}$ by Lemma \ref{lemproprQuvx}, iii).
    $$\theta_{u,v,x}(\tr_{W_1}^W(f))(\mathfrak{a})=\theta_{u,v,x}(\sum_{W/W_1}wf(w^{-1}\mathfrak{a}))=\sum_{W/W_1}f(w^{-1}\mathfrak{a})s^{-1}(uxv^{-1}).$$

    \item[iii)] Let $\alpha\in k$. It follows that
    \begin{align*}
      (\pi_{u,v,x}\circ \theta_{u',v',y})(\alpha)&=\pi_{u,v,x}(\alpha s^{-1}(u'y(v')^{-1}))\\
      &=\alpha\delta_{u,v,x}[\pi_x(s(s^{-1}(u'y(v')^{-1})))]\\
      &=\left\{ \begin{array}{ll}
\alpha\delta_{u,v,x}(u'y(v')^{-1}),\quad \text{if}~x=y\\
0,\qquad~~~~~~~~~~~~~~~~~~~ \text{if}~x\neq y
\end{array} \right.
\\
      &=\left\{ \begin{array}{ll}
\alpha,\quad \text{if}~x=y~\text{and}~(u',v')\in Q_{u,v,x}(u,v)Q_x\\
0,\quad\text{otherwise}
\end{array} \right.
    \end{align*}
    The particular case is an immediate consequence.
  \end{itemize}
\end{proof}
\begin{lemma}\label{lemthetthetcup} With the notations of (\ref{eqs-1s-1}) and(\ref{equ'zv'z}) the following statement holds
$$\theta^*_{u_x,v_x}([f])\cup\theta^*_{wu_y,wv_y}([g])=\sum_{\substack{z\in Z\\(u'_z,v'_z)\in A_z}}\theta_{u'_z,v'_z}^*([f]\cup [g])$$
for any $[f]\in\Hc^*(W',k)$ and $[g]\in\Hc^*(W',k).$
\end{lemma}
\begin{proof}
Let $D:\mathcal{P}_*\rightarrow\mathcal{P}_*\otimes\mathcal{P}_*$ be the diagonal approximation, where $\mathcal{P}_*\rightarrow k$ is a $k\Delta P$-projective resolution of the trivial $k\Delta P$-module $k$. By the definition of the cup product we have
$$\theta^*_{u_x,v_x}([f])\cup \theta_{wu_y,wv_y}^*([g])=[\mu\circ (\theta_{u_x,v_x}\otimes \theta_{wu_y,wv_y})\circ (f\otimes g)\circ D],$$
where $\mu$ is the multiplication map in the source algebra. But since, for any $\alpha\in k$
$$(\mu\circ (\theta_{u_x,v_x}\otimes\theta_{wu_y,wv_y}))(\alpha)=\alpha s^{-1}(u_xxv_x^{-1})s^{-1}(wu_yyv_y^{-1}w^{-1})$$
$$=\sum_{\substack{z\in Z\\(u'_z,v'_z)}}\alpha\alpha_zs^{-1}(u'_zz(v'_z)^{-1})=\sum_{\substack{z\in Z\\(u'_z,v'_z)\in A_z}}\alpha_z\theta_{u'_z,v'_z}(\alpha),$$
we are done.
\end{proof}

\emph{Proof of Theorem \ref{thmprodformula}.} We apply \cite[Propostion 4.2.4, Theorem 4.2.6]{Ev} to obtain the following equalities
\begin{align*}
&~\gamma_{u_x,v_x}([f])\cup \gamma_{u_y,v_y}([g])\\&=\tr_{Q_{u_x,v_x}}^{\Delta P}(\theta_{u_x,v_x}^*([f]))\cup \tr_{Q_{u_y,v_y}}^{\Delta P}(\theta_{u_y,v_y}^*([g]))\\
&=\tr_{Q_{u_x,v_x}}^{\Delta P}\left[\theta_{u_x,v_x}^*([f])\cup(\res^{\Delta P}_{Q_{u_x,v_x}}\circ \tr^{\Delta P}_{Q_{u_y,v_y}})(\theta_{u_y,v_y}^*([g]))\right]\\
&=\sum_{w\in D}\tr_{Q_{u_x,v_x}}^{\Delta P}\left[\theta_{u_x,v_x}^*([f])\cup(\tr^{Q_{u_x,v_x}}_{{}^wQ_{u_y,v_y}\cap Q_{u_x,v_x}}\circ \right. \\
&\left.\circ\res^{{}^wQ_{u_y,v_y}}_{{}^wQ_{u_y,v_y}\cap Q_{u_x,v_x}})(w^*(\theta_{u_y,v_y}^*([g])))\right]
\end{align*}
By Lemma \ref{lemproprthetcircpi}, Lemma \ref{lemthetthetcup} and again \cite[Proposition 4.2.4]{Ev} the above sum is
\begin{align*}
&\sum_{w\in D}\tr_{Q_{u_x,v_x}}^{\Delta P}\left[\theta_{u_x,v_x}^*([f])\cup(\tr^{Q_{u_x,v_x}}_{W'}\circ \res^{Q_{wu_y,wv_y}}_{W'})(\theta_{wu_y,wv_y}^*(w^*([g])))\right]\\
&=\sum_{w\in D}\tr_{W'}^{\Delta P}\left[\res^{Q_{u_x,v_x}}_{W'}(\theta_{u_x,v_x}^*([f]))\cup \res^{Q_{wu_y,wv_y}}_{W'}(\theta_{wu_y,wv_y}^*(w^*([g])))\right]\\
&=\sum_{w\in D}\tr_{W'}^{\Delta P}\left[\res^{Q_{u_x,v_x}}_{W'}(\theta_{u_x,v_x}^*([f]))\cup \res^{Q_{wu_y,wv_y}}_{W'}(\theta_{wu_y,wv_y}^*(w^*([g])))\right]\\
&=\sum_{w\in D}\tr_{W'}^{\Delta P}\left[\theta_{u_x,v_x}^*(\res^{Q_{u_x,v_x}}_{W'}([f]))\cup \theta_{wu_y,wv_y}^*(\res^{Q_{wu_y,wv_y}}_{W'}(w^*([g])))\right]\\
&=\sum_{w\in D}\sum_{\substack{z\in Z\\(u'_z,v'_z)\in A_z}}\alpha_z\left(\tr_{W'}^{\Delta P} \circ \theta_{u'_z,v'_z}^*\right)\left(\res^{Q_{u_x,v_x}}_{W'}([f])\cup \res^{{}^wQ_{u_y,v_y}}_{W'}(w^*([g]))\right)
\end{align*}
Next we apply the isomorphism from Theorem \ref{thmadditivdecom}, Mackey decomposition 
\cite[Theorem 4.2.6]{Ev} and Lemma \ref{lemproprthetcircpi} i), ii), hence the above sum becomes
\begin{align*}
  &\sum_{\substack{w\in D\\z\in Z\\(u'_z,v'_z)\in A_z}}\sum_{\substack{t\in Y_{G,b,P_{\delta}}\\(u_t,v_t)\in [Q_t]}}\alpha_z(\gamma_{u_t,v_t}\circ\pi_{u_t,v_t}^*\circ\res^{\Delta P}_{Q_{u_t,v_t}}\circ \tr^{\Delta P}_{W'}\circ\theta^*_{u'_z,v'_z})(\res^{Q_{u_x,v_x}}_{W'}([f])\cup\\&\cup \res^{Q_{wu_y,wv_y}}_{W'}(w^*([g]))\\&=
  \sum_{w\in D}\sum_{\substack{z\in Z\\(u'_z,v'_z)\in A_z}}\alpha_z\sum_{\substack{t\in Y_{G,b,P_{\delta}}\\(u_t,v_t)\in [Q_t]}}\sum_{r\in[Q_{u_t,v_t}\backslash\Delta P/W']}(\gamma_{u_t,v_t}\circ\pi_{u_t,v_t}^*\circ\tr^{Q_{u_t,v_t}}_{Q_{u_t,v_t}\cap{}^r W'}\circ \\&\circ \res^{{}^rW'}_{Q_{u_t,v_t}\cap{}^rW'}\circ\theta^*_{ru'_z,rv'_z})(\res^{{}^rQ_{u_x,v_x}}_{{}^rW'}(r^*[f])\cup \res^{{}^rQ_{wu_y,wv_y}}_{{}^rW'}((rw)^*([g]))\\&=
  \sum_{w\in D}\sum_{\substack{z\in Z\\(u'_z,v'_z)\in A_z}}\alpha_z\sum_{\substack{t\in Y_{G,b,P_{\delta}}\\(u_t,v_t)\in [Q_t]}}\sum_{r\in[Q_{u_t,v_t}\backslash\Delta P/W']}(\gamma_{u_t,v_t}\circ\pi_{u_t,v_t}^*\circ\theta^*_{ru'_z,rv'_z}\circ\\&\circ\tr^{Q_{u_t,v_t}}_{Q_{u_t,v_t}\cap{}^r W'})(\res^{{}^rQ_{u_x,v_x}}_{Q_{u_t,v_t}\cap{}^rW'}(r^*[f])\cup \res^{{}^rQ_{wu_y,wv_y}}_{Q_{u_t,v_t}\cap{}^rW'}((rw)^*([g])).
\end{align*}

By Lemma \ref{lemproprthetcircpi}, iii) each $t$ must be equal to some $z$ to obtain non-zero terms, and in these cases by (\ref{equ'zv'z}) we know that
$$(ru'_z,rv'_z)=(u_za,v_za^z)$$
for some $a\in Q_z$. We obtain the shorter sum
 $$ \sum_{\substack{w\in D\\z\in Z\\(u_z,v_z)}}\alpha_z(\gamma_{u_z,v_z}\circ\pi_{u_z,v_z}^*\circ\theta^*_{u_za,v_za^z}\circ\tr^{Q_{u_z,v_z}}_{Q_{u_z,v_z}\cap{}^r W'})(\res^{{}^rQ_{u_x,v_x}}_{Q_{u_z,v_z}\cap{}^rW'}(r^*[f])\cup$$$$\cup \res^{{}^rQ_{wu_y,wv_y}}_{Q_{u_z,v_z}\cap{}^rW'}((rw)^*([g]))$$$$=\sum_{\substack{w\in D\\z\in Z\\(u_z,v_z)}}\alpha_z\gamma_{u_z,v_z}(\tr^{Q_{u_z,v_z}}_{Q_{u_z,v_z}\cap{}^r W'}(\res^{{}^rQ_{u_x,v_x}}_{Q_{u_z,v_z}\cap{}^rW'}(r^*[f])\cup \res^{{}^{rw}Q_{u_y,v_y}}_{Q_{u_z,v_z}\cap{}^rW'}((rw)^*([g]))).$$
 The last equality is true since we apply Proposition \ref{propproprtetaapigamma}, i) and Lemma \ref{lemproprthetcircpi}, iii). We are done.

 \section{Bockstein homomorphisms for cohomology of defect groups with coefficients in source algebras and further properties}\label{sec5}

We continue with the notations of Section \ref{sec4}. The additive decomposition from Theorem \ref{thmadditivdecom} allows us to obtain a similar definition as in Remark \ref{remexplbockhochsch} for Bockstein homomorphisms of $\Hc^*(\Delta P, ikGi)$.
\begin{definition}\label{defnBockikgi}
  The \emph{Bockstein homomorphism for the cohomology of a defect group $P$ with coefficients in the source algebra} $ikGi$ is the map
  $$\bbbeta_{P,ikGi}:\Hc^r(\Delta P, ikGi)\rightarrow \Hc^{r+1}(\Delta P, ikGi)$$
  $$\bbbeta_{P,ikGi}(\gamma_{u_x,v_x}([f])):=\gamma_{u_x,v_x}(\beta_{Q_{u_x,v_x}}([f])),$$
  for any $x\in Y_{G,b,P_{\delta}}, (u_x,v_x)\in[Q_x]$ and $[f]\in\Hc^r(Q_{u_x,v_x},k)$.
\end{definition}
Some consequences of Theorem \ref{thmprodformula} are included in the following theorem where we collect some properties of $\bbbeta_{P,ikGi}$.
\begin{theorem} \label{thmbocksource}With the above notations the following statements hold:
\begin{itemize}
  \item[i)] $\bbbeta_{P,ikGi}$ is a differential;
  \item[ii)] The Bockstein map $\bbbeta_{P,ikGi}$ is a graded derivation, hence $\Hc^*(\Delta P,ikGi)$ is a DG-algebra;
  \item[iii)] The Bockstein map $\bbbeta_{P,ikGi}$  is compatible to $\beta_{G,b}$, i.e. the next diagram is commutative
      \begin{displaymath}
 \xymatrix{\Hc^r(G,b,P_{\delta})\ar[rr]^{\beta_{G,b}}\ar[d]^{\gamma_{u_1,v_1}}&&\Hc^{r+1}(G,b,P_{\delta})\ar[d]^{\gamma_{u_1,v_1}} \\ \Hc^r(\Delta P,ikGi)\ar[rr]^{\bbbeta_{P,ikGi}}&&\Hc^{r+1}(\Delta P,ikGi)
                                            }.
\end{displaymath}
\end{itemize}

\end{theorem}
\begin{proof}
In order to prove statements i) and ii) we apply similar methods as in the proof of Theorem \ref{thmBockderiv}; for showing these statements Theorem \ref{thmprodformula} is a crucial ingredient. By Definition \ref{defnBockikgi} we know that
$$\bbbeta_{P,ikGi}(\gamma_{u_1,v_1}([f])):=\gamma_{u_1,v_1}(\beta_{Q_{u_1,v_1}}([f])),$$
  for any $[f]\in\Hc^r(Q_{u_1,v_1},k)$. But by Remark \ref{remgammau1} we  choose $(u_1,v_1)=(1,1)\in\Delta P$ such that $Q_{1,1}=\Delta P$ and $\gamma_{u_1,v_1}=\theta_{1,1}^r$ hence
  $$\gamma_{u_1,v_1}:\Hc^r(\Delta P,k)\rightarrow\Hc^{r}(\Delta P,ikGi).$$
Next, Remark \ref{remdefnbockcohombl} implies that if $[f]$ is in $\Hc^r(G,b,P_{\delta})$ then $\beta_{G,b}([f])=\beta_P([f])$ is in $\Hc^{r+1}(G,b,P_{\delta})$. 
\end{proof}
 Theorem \ref{thmprodformula} and Bockstein map  defined in Definition \ref{defnBockikgi} can be applied in order to compute the graded algebra structure of the cohomology algebra $\Hc^*(\Delta P, ikGi)$ for some particular cases, for example when $P$ is cyclic of order $p$. This can be done similarly to \cite[Example p.28]{Ev}, but we do not pursue this here.

\end{document}